\documentclass[12pt,reqno]{article}

\usepackage[usenames]{color}
\usepackage[colorlinks=true,
linkcolor=webgreen, filecolor=webbrown,
citecolor=webgreen]{hyperref}

\definecolor{webgreen}{rgb}{0,.5,0}
\definecolor{webbrown}{rgb}{.6,0,0}

\usepackage{amssymb}
\usepackage{graphicx}
\usepackage{amscd}
\usepackage{lscape}
\usepackage{tikz}
\usepackage{tikz-cd}
\usepackage{pgfplots}

\usetikzlibrary{matrix}
\usetikzlibrary{fit,shapes}
\usetikzlibrary{positioning, calc}
\tikzset{circle node/.style = {circle,inner sep=1pt,draw, fill=white},
        X node/.style = {fill=white, inner sep=1pt},
        dot node/.style = {circle, draw, inner sep=5pt}
        }
\usepackage{tkz-fct}

\usepackage{amsthm}
\newtheorem{theorem}{Theorem}

\newtheorem{proposition}[theorem]{Proposition}
\newtheorem{corollary}[theorem]{Corollary}

\theoremstyle{definition}

\usepackage{float}

\usepackage{graphics,amsmath}
\usepackage{amsfonts}
\usepackage{latexsym}
\usepackage{epsf}

\newcommand{\seqnum}[1]{\href{http://oeis.org/#1}{\underline{#1}}}

\begin{document}

\begin{center}
\vskip 1cm{\LARGE\bf The $1/k$-Eulerian Polynomials as Moments, via Exponential Riordan Arrays} \vskip 1cm \large
Paul Barry\\
School of Science\\
Waterford Institute of Technology\\
Ireland\\
\href{mailto:pbarry@wit.ie}{\tt pbarry@wit.ie}
\end{center}
\vskip .2 in

\begin{abstract} Using the theory of exponential Riordan arrays, we show that the $1/k$-Eulerian polynomials are moments for a paramaterized family of orthogonal polynomials. In addition, we show that the related Savage-Viswanathan polynomials are also moments for appropriate families of orthogonal polynomials. We provide continued fraction ordinary generating functions and Hankel transforms for these moments, as well as the three-term recurrences for the corresponding orthogonal polynomials. We provide formulas for the $1/k$-Eulerian polynomials and the Savage-Viswanathan polynomials involving the Stirling numbers of the first and the second kind. Finally we show that the once-shifted polynomials are again moment sequences. \end{abstract}

\section{Introduction}

The $1/k$-Eulerian polynomials have been defined and characterized by Savage and Viswanathan \cite{Savage}. In doing so, they showed that these polynomials are specializations of bivariate polynomials $F_n(x,y)$ that have been studied previously by Carlitz \cite{Carlitz}, by Dillon and Roselle \cite{Dillon}, and by Foata and Schutzenberger \cite{Foata}. In this note, we shall call these polynomials the Savage-Viswanathan polynomials, given the depth of coverage that these two authors afford them.

The Savage-Viswanathan polynomials can be defined by their exponential generating function, which is given by
\begin{equation}\label{SV} \sum_{n=0}^{\infty} F_n(x,y) \frac{z^n}{n!} = \left(\frac{1-x}{e^{z(x-1)}-x}\right)^y. \end{equation}

They begin
$$1, y, y(x + y), y(x^2 + x(3y + 1) + y^2), y(x^3 + x^2(7y + 4) + x(6y^2 + 4y + 1) + y^3),\ldots.$$
The authors Savage and Viswanathan \cite{Savage} show that
$$F_n(x,y)=\sum_{\pi \in S_n} x^{exc(\pi)}y^{\#cyc(\pi)},$$
where $S_n$ is the set of permutations $\pi:\{1,2,\ldots,n\} \to \{1,2,\ldots,n\}$, $exc(\pi)=\#\{i|\pi(i)>i\}$, and $\#cyc(\pi)$ is the number of cycles in the disjoint cycle representation of $\pi$.

We then have
$$A_n^{(k)}(x) = k^n F_n(x,1/k)$$ where $A_n^{(k)}(x)$ denotes the $n$-th $1/k$-Eulerian polynomial. The exponential generating function for these polynomials is then \cite{Savage}
$$  \sum_{n=0}^{\infty} A_n^{(k)}(x)\frac{z^n}{n!}=\left(\frac{1-x}{e^{kx(x-1)}-x}\right)^{1/k}.$$

They begin
$$1, 1, k x + 1, k^2x^2 + kx(k + 3) + 1, k^3x^3 + k^2x^2(4k + 7) + kx(k^2 + 4·k + 6) + 1,\ldots.$$

In a previous note \cite{Eulerian}, we have shown that the Eulerian polynomials are moments for a family of orthogonal polynomials whose coefficient array is an exponential Riordan array. In this current note, we wish to prove a similar result for both the $1/k$-Eulerian polynomials and the Savage-Viswanathan polynomials. Thus we wish to prove the following results.

\begin{proposition} The $1/k$-Eulerian polynomials are the moments for a family of orthogonal polynomials whose coefficient array is given by the exponential Riordan array
$$\left[\frac{1}{(1+kz)^{1/k}}, \frac{1}{k(1-x)} \ln\left(\frac{1+kz}{1+kxz}\right)\right].$$ The $1/k$-Eulerian polynomials then appear as the initial column of the inverse array
$$\left[\left(\frac{1-x}{e^{kz(x-1)}-x}\right)^{1/k}, \frac{e^{kz}-e^{kxz}}{k(e^{kxz}-x e^{kz})}\right].$$
\end{proposition}
\begin{corollary}
The $1/k$-Eulerian polynomials are the moments for the family of orthogonal polynomials $P_n^{(k)}(z)$ defined by the three-term recurrence
$$P_n^{(k)}(z)=(z-(k(n-1)(1+x)+1))P_{n-1}^{(k)}(z)-(k^2 x(n^2-3n+2)+kx(n-1))P_{n-2}^{(k)}(z),$$ with
$P_0^{(k)}(z)=1$ and $P_1^{(k)}(z)=z-1$.
\end{corollary}
\begin{corollary} The $1/k$-Eulerian polynomials have their ordinary generating function given by the continued fraction
$$\cfrac{1}{1-
\cfrac{z}{1-
\cfrac{kxz}{1-
\cfrac{(k+1)z}{1-
\cfrac{2kxz}{1-
\cfrac{(2k+1)z}{1-
\cfrac{3kxz}{1-
\cfrac{(3k+1)z}{1-\cdots}}}}}}}},$$ or equivalently
$$\cfrac{1}{1-z-
\cfrac{kxz^2}{1-(kx+k+1)z-
\cfrac{2kx(k+1)z^2}{1-(2kx+2k+1)z-
\cfrac{3kx(2k+1)z^2}{1-(3kx+3k+1)z-\cdots}}}}.$$
\end{corollary}
Note that in Del\'eham notation \cite{CFT} (and see Del\'eham's remarks in \seqnum{A084938}), this first continued fraction corresponds to
$$[1,0,k+1,0,2k+1,0,3k+1,0,\ldots]\,\Delta\,[0,k,0,2k,0,3k,0,\ldots].$$
\begin{corollary} The Hankel transform of the $1/k$-Eulerian polynomials is given by
$$h_n=(kx)^{\binom{n+1}{2}}(\prod_{j=0}^n j!)\prod_{i=1}^n (ik+1)^{n-i}=(kx)^{\binom{n+1}{2}} \prod_{i=1}^n i(ik+1)^{n-i}.$$
\end{corollary}
\begin{proposition} The Savage-Viswanathan polynomials $F_n(x,y)$ are the moments for a family of orthogonal polynomials whose coefficient array is given by the exponential Riordan array
$$\left[\frac{1}{(1+z)^y}, \frac{1}{1-x} \ln\left(\frac{1+z}{1+xz}\right)\right].$$
 The Savage-Viswanathan polynomials then appear as the initial column of the inverse array
$$\left[ \left(\frac{1-x}{e^{z(x-1)}-x}\right)^y, \frac{e^{z}-e^{xz}}{e^{xz}-x e^{z}}\right].$$
\end{proposition}
Note that we can write the last exponential Riordan array as
$$\left[e^{y z} \left(\frac{1-x}{e^{xz}-x e^{z}}\right)^y, \frac{e^{z}-e^{xz}}{e^{xz}-x e^{z}}\right].$$
\begin{corollary}
The Savage-Viswanathan polynomials $F_n(x,y)$ are the moments for the family of orthogonal polynomials $Q_n^{(x,y)}(z)$ that are defined by the three-term recurrence
$$Q_n^{(x,y)}(z)=(z-(y+(n-1)(1+x)))Q_{n-1}^{(x,y)}(z)-(n-1)(xy+(n-2)x)Q_{n-2}^{(x,y)}(z),$$
with $Q_0^{(x,y)}(z)=1$ and $Q_1^{(x,y)}(z)=z-y$.
\end{corollary}

\begin{corollary} The Savage-Viswanathan polynomials have their ordinary generating function given by the continued fraction
$$\cfrac{1}{1-
\cfrac{yz}{1-
\cfrac{xz}{1-
\cfrac{(y+1)z}{1-
\cfrac{2xz}{1-
\cfrac{(y+2)z}{1-
\cfrac{3xz}{1-
\cfrac{(y+3)z}{1-\cdots}}}}}}}},$$ or equivalently
$$\cfrac{1}{1-yz-
\cfrac{xyz^2}{1-(x+y+1)z-
\cfrac{2x(y+1)z^2}{1-(2x+y+2)z-
\cfrac{3x(y+2)z^2}{1-(3x+y+3)z-\cdots}}}}.$$
\end{corollary}
Note that in Del\'eham notation, the first continued fraction above corresponds to
$$[y,0,y+1,0,y+2,0,y+3,0,\ldots]\,\Delta\,[0,1,0,2,0,3,0,\ldots].$$
\begin{corollary} The Hankel transform of the Savage-Viswanathan polynomials is given by
$$h_n = x^{\binom{n+1}{2}}y^n \prod_{i=1}^n ((i+1)(i+y))^{n-i}.$$
\end{corollary}

\section{Background material}
In this section, we give some background information on exponential Riordan arrays, orthogonal polynomials and Hankel transforms.

 The \emph{exponential Riordan group} \cite
{Barry_Pascal, DeutschShap, ProdMat}, is a set of
infinite lower-triangular integer matrices, where each matrix
is defined by a pair
of generating functions $g(x)=g_0+g_1x+g_2x^2+\cdots$ and
$f(x)=f_1x+f_2x^2+\cdots$ where $g_0 \ne 0$ and $f_1\ne 0$. We usually assume that
$$g_0=f_1=1.$$
The associated
matrix is the matrix
whose $i$-th column has exponential generating function
$g(x)f(x)^i/i!$ (the first column being indexed by $0$). The
matrix corresponding to
the pair $f, g$ is denoted by $[g, f]$.  The group law is given by \begin{displaymath}
[g,
f]\cdot [h,
l]=[g(h\circ f), l\circ f].\end{displaymath} The identity for
this law is $I=[1,x]$ and the inverse of $[g, f]$ is $[g,
f]^{-1}=[1/(g\circ
\bar{f}), \bar{f}]$ where $\bar{f}$ is the compositional
inverse of $f$.

If $\mathbf{M}$ is the matrix $[g,f]$, and
$\mathbf{u}=(u_n)_{n \ge 0}$
is an integer sequence with exponential generating function
$\mathcal{U}$
$(x)$, then the sequence $\mathbf{M}\mathbf{u}$ has
exponential
generating function $g(x)\mathcal{U}(f(x))$. This result is known as the fundamental theorem of (exponential) Riordan arrays. Thus the row sums
of the array
$[g,f]$ have exponential generating function given by $g(x)e^{f(x)}$ since the sequence
$1,1,1,\ldots$ has exponential generating function $e^x$.

As an element of the group of exponential Riordan arrays, the binomial matrix $\mathbf{B}$ with $(n,k)$-th element $\binom{n}{k}$ is given by
 $\mathbf{B}=[e^x,x]$. By the above, the exponential
generating function of
its row sums is given by $e^x e^x = e^{2x}$, as expected
($e^{2x}$ is the e.g.f. of $2^n$).

To each exponential Riordan array $L=[g,f]$ is associated \cite{ProdMat_0, ProdMat} a matrix $P$, called its \emph{production} matrix, which has
bivariate generating function given by
$$e^{xy}(Z(x)+A(x)y)$$ where
$$A(x)=f'(\bar{f}(x)), \quad Z(x)=\frac{g'(\bar{f}(x))}{g(\bar{f}(x))}.$$
We have $$P=L^{-1}\bar{L}$$ where $\bar{L}$ \cite{PPWW, Wall} is the matrix $L$ with its top row removed.

In order to demonstrate our results, we will require three results from the theory of exponential Riordan arrays. These are \cite{Classical, Barry_Meixner, Barry_Moment}
\begin{enumerate}
\item The inverse of an exponential Riordan array $[g,f]$ is the coefficient array of a family of orthogonal polynomials if and only if the production matrix of $[g,f]$ is tri-diagonal;
\item If the production matrix \cite{DeutschShap, ProdMat_0, ProdMat} of $[g,f]$ is tri-diagonal, then the elements of the first column of $[g,f]$ are the moments of the corresponding family of orthogonal polynomials;
\item The bivariate generating function of the production matrix of $[g,f]$ is given by
$$e^{xy}(Z(x)+A(x)y)$$ where
$$A(x)=f'(\bar{f}(x)),$$ and
$$Z(x)=\frac{g'(\bar{f}(x))}{g(\bar{f}(x))},$$ where $\bar{f}(x)$ is the compositional inverse (series reversion) of
$f(x)$.
\end{enumerate}

General information on Riordan arrays is provided in \cite{Book}.

For general information on orthogonal polynomials and moments, see \cite{Chihara, Gautschi, Szego}. Continued fractions will be referred to in the sequel; \cite{Wall} is a general reference, while \cite{Kratt, Kratt1} discuss the connection between continued fractions and orthogonal polynomials, and moments and Hankel transforms \cite{Layman, Radoux}. We recall that for a given sequence $a_n$ its Hankel transform is the sequence of determinants $h_n=|a_{i+j}|_{0 \le i,j \le n}$.

Many interesting examples of number triangles, including exponential Riordan arrays, can be found in Neil Sloane's On-Line
Encyclopedia of Integer Sequences \cite{SL1, SL2}. Sequences are frequently referred to by their
OEIS number. For instance, the binomial matrix (Pascal's triangle) $\mathbf{B}$ with $(n,k)$-th element $\binom{n}{k}$ is \seqnum{A007318}.

The following well-known results (the first is the well-known ``Favard's Theorem''), which we essentially reproduce from
\cite{Kratt}, specify the links between orthogonal polynomials, the three-term recurrences that define them, the recurrence coefficients of those three-term recurrences, and the g.f. of
the moment sequence of the orthogonal polynomials.
\begin{theorem} \label{ThreeT}\cite{Kratt} (Cf. \cite[Th\'eor\`eme 9 on p.I-4]{Viennot}, or \cite[Theorem 50.1]{Wall}). Let $(p_n(x))_{n\ge 0}$
be a sequence of monic polynomials, the polynomial $p_n(x)$ having degree $n=0,1,\ldots$ Then the sequence $(p_n(x))$ is (formally)
orthogonal if and only if there exist sequences $(\alpha_n)_{n\ge 0}$ and $(\beta_n)_{n\ge 1}$ with $\beta_n \neq 0$ for all $n\ge 1$,
such that the three-term recurrence
$$p_{n+1}=(x-\alpha_n)p_n(x)-\beta_n p_{n-1}(x), \quad \text{for}\quad n\ge 1, $$
holds, with initial conditions $p_0(x)=1$ and $p_1(x)=x-\alpha_0$.
\end{theorem}

\begin{theorem} \label{CF} \cite{Kratt} (Cf. \cite[Proposition 1, (7), on p. V-5]{Viennot}, or \cite[Theorem 51.1]{Wall}). Let $(p_n(x))_{n\ge
0}$ be a sequence of monic polynomials, which is orthogonal with respect to some functional $\mathcal{L}$. Let
$$p_{n+1}=(x-\alpha_n)p_n(x)-\beta_n p_{n-1}(x), \quad \text{for}\quad n\ge 1, $$ be the corresponding three-term recurrence which is
guaranteed by Favard's theorem. Then the generating function
$$g(x)=\sum_{k=0}^{\infty} \mu_k x^k $$ for the moments $\mu_k=\mathcal{L}(x^k)$ satisfies
$$g(x)=\cfrac{\mu_0}{1-\alpha_0 x-
\cfrac{\beta_1 x^2}{1-\alpha_1 x -
\cfrac{\beta_2 x^2}{1-\alpha_2 x -
\cfrac{\beta_3 x^2}{1-\alpha_3 x -\cdots}}}}.$$
\end{theorem}
\noindent
The {\it
Hankel
transform} \cite{Layman} of a given sequence
$A=\{a_0,a_1,a_2,...\}$ is the
sequence of Hankel determinants $\{h_0, h_1, h_2,\dots \}$
where
$h_{n}=|a_{i+j}|_{i,j=0}^{n}$, i.e

\begin{center} \begin{equation}
 \label{gen1}
 A=\{a_n\}_{n\in\mathbb N_0}\quad \rightarrow \quad
 h=\{h_n\}_{n\in\mathbb N_0}:\quad
h_n=\left| \begin{array}{ccccc}
 a_0\ & a_1\  & \cdots & a_n  &  \\
 a_1\ & a_2\  &        & a_{n+1}  \\
\vdots &      & \ddots &          \\
 a_n\ & a_{n+1}\ &    & a_{2n}
\end{array} \right|. \end{equation} \end{center} The Hankel
transform of a sequence $a_n$ and that of its binomial transform are
equal.

\noindent In the case that $a_n$ has g.f. $g(x)$ expressible in the form
$$g(x)=\cfrac{a_0}{1-\alpha_0 x-
\cfrac{\beta_1 x^2}{1-\alpha_1 x-
\cfrac{\beta_2 x^2}{1-\alpha_2 x-
\cfrac{\beta_3 x^2}{1-\alpha_3 x-\cdots}}}}$$ then
we have \cite{Kratt}
\begin{equation}\label{Kratt} h_n = a_0^{n+1} \beta_1^n\beta_2^{n-1}\cdots \beta_{n-1}^2\beta_n=a_0^{n+1}\prod_{k=1}^n
\beta_k^{n+1-k}.\end{equation}
Note that this is independent of $\alpha_n$.

\section{The proofs}
We start with the exponential array for the $1/k$-Eulerian polynomials,
$$[g,f]=\left[\left(\frac{1-x}{e^{kz(x-1)}-x}\right)^{1/k}, \frac{e^{kz}-e^{kxz}}{k(e^{kxz}-x e^{kz})}\right].$$
Here,
$$f(z)=\frac{e^{kz}-e^{kxz}}{k(e^{kxz}-x e^{kz})}$$ and
$$\bar{f}(z)=\frac{1}{k(1-x)} \ln\left(\frac{1+kz}{1+kxz}\right).$$
Now $$f'(x)=\frac{(1-x)^2 e^{kz(1+x)}}{(e^{kxz}-xe^{kz})^2},$$ and hence we have
$$A(z)=f'(\bar{f}(z))=(1+kz)(1+kxz).$$
We have $$g(z)=\left(\frac{1-x}{e^{kz(x-1)}-x}\right)^{1/k},$$ and
hence
$$g'(z)=e^{z(1+kx)}\left(\frac{1-x}{e^{kxz}-xe^{kz}}\right)^{(k+1)/k}.$$
Then
$$Z(z)=\frac{g'(\bar{f}(z))}{g(\bar{f}(z))}=1+kxz.$$
Thus the production matrix, which has bivariate generating function given by
$$ e^{zw}(1+kxz+w(1+kz)(1+kxz))$$ is tri-diagonal.
The production matrix begins
$$\left(
\begin{array}{ccccccc}
 1 & 1 & 0 & 0 & 0 & 0 & 0 \\
 k x & k x+k+1 & 1 & 0 & 0 & 0 & 0 \\
 0 & 2 k^2 x+2 k x & 2 k x+2 k+1 & 1 & 0 & 0 & 0 \\
 0 & 0 & 6 k^2 x+3 k x & 3 k x+3 k+1 & 1 & 0 & 0 \\
 0 & 0 & 0 & 12 k^2 x+4 k x & 4 k x+4 k+1 & 1 & 0 \\
 0 & 0 & 0 & 0 & 20 k^2 x+5 k x & 5 k x+5 k+1 & 1 \\
 0 & 0 & 0 & 0 & 0 & 30 k^2 x+6 k x & 6 k x+6 k+1 \\
\end{array}
\right).$$
The three-term recurrence, continued fractions and Hankel transforms now follow from these coefficients.

We now turn to the moment matrix for the Savage-Viswanathan polynomials,
$$[g,f]=\left[ \left(\frac{1-x}{e^{z(x-1)}-x}\right)^y, \frac{e^{z}-e^{xz}}{e^{xz}-x e^{z}}\right].$$
We have
$$f(x)=\frac{e^{z}-e^{xz}}{e^{xz}-x e^{z}},$$ and hence
$$f'(x)=\frac{(1-x)^2e^{z(1+x)}}{(e^{xz}-xe^z)^2},$$
and
$$\bar{f}(x)= \frac{1}{1-x} \ln\left(\frac{1+z}{1+xz}\right).$$
We find that
$$A(z)=f'(\bar{f}(x))=(1+z)(1+xz).$$
Now $$g(x)= \left(\frac{1-x}{e^{z(x-1)}-x}\right)^y,$$ and hence
$$g'(x)=y e^{z(x+y)}\left(\frac{1-x}{e^{xz}-xe^z}\right)^{y+1}.$$
Hence
$$Z(z)=\frac{g'(\bar{f}(z))}{g(\bar{f}(z))}=y(1+xz).$$
Thus the production matrix, which has bivariate generating function given by
$$e^{zw}(y(1+xz)+w(1+z)(1+xz))$$ is tri-diagonal.
The production matrix then begins
$$\left(
\begin{array}{ccccccc}
 y & 1 & 0 & 0 & 0 & 0 & 0 \\
 x y & x+y+1 & 1 & 0 & 0 & 0 & 0 \\
 0 & 2 (x y+x) & 2 x+y+2 & 1 & 0 & 0 & 0 \\
 0 & 0 & 3 x y+6 x & 3 x+y+3 & 1 & 0 & 0 \\
 0 & 0 & 0 & 4 x y+12 x & 4 x+y+4 & 1 & 0 \\
 0 & 0 & 0 & 0 & 5 x y+20 x & 5 x+y+5 & 1 \\
 0 & 0 & 0 & 0 & 0 & 6 x y+30 x & 6 x+y+6 \\
\end{array}
\right).$$
The three-term recurrence, continued fractions and Hankel transforms now follow from these coefficients.

\section{Some remarks}
Using the fundamental theory of exponential Riordan arrays, we have

$$ \left(\frac{1-x}{e^{z(x-1)}-x}\right)^y=\left[1, z-\ln\left(\frac{e^{xz}-ze^x}{1-x}\right)\right]\cdot e^{yz}.$$
The expansion of $e^{yz}$ is the monomial sequence ${1,y,y^2,\ldots}$ and hence this says that when we expand the Savage-Viswanathan polynomials in $y$, the coefficient matrix is the exponential Riordan array
$$\left[1, z-\ln\left(\frac{e^{xz}-xe^z}{1-x}\right)\right]=\left[1, \frac{e^{z(x-1)}-1}{x-1}\right]\cdot \left[1, \ln\left(\frac{1}{1-x}\right)\right].$$ In this product of exponential Riordan arrays, the first matrix is the matrix of generalized Stirling numbers of the second kind with parameter $x-1$, while the second matrix represents the unsigned Stirling numbers of the first kind. Thus we have the following formula for the Savage-Visnawathan polynomials based on the Stirling numbers of the first and second kinds.
$$F_n(x,y)=\sum_{k=0}^n \sum_{j=0}^n S_{n,j}(x-1)^{n-j} |s_{j,k}|y^k,$$ where $S_{n,k}$ are the Stirling numbers of the second kind and $|s_{n,k}|$ are the unsigned Stirling numbers of the first kind.

The matrix $\left[1, z-\ln\left(\frac{e^{xz}-xe^z}{1-x}\right)\right]$ begins
\begin{scriptsize}
$$\left(
\begin{array}{ccccccc}
 1 & 0 & 0 & 0 & 0 & 0 & 0 \\
 0 & 1 & 0 & 0 & 0 & 0 & 0 \\
 0 & x & 1 & 0 & 0 & 0 & 0 \\
 0 & x^2+x & 3 x & 1 & 0 & 0 & 0 \\
 0 & x^3+4 x^2+x & 7 x^2+4 x & 6 x & 1 & 0 & 0 \\
 0 & x^4+11 x^3+11 x^2+x & 15 x^3+30 x^2+5 x & 25 x^2+10 x & 10 x & 1 & 0 \\
 0 & x^5+26 x^4+66 x^3+26 x^2+x & 31 x^4+146 x^3+91 x^2+6 x & 90 x^3+120 x^2+15 x & 65 x^2+20 x & 15 x & 1 \\
\end{array}
\right).$$
\end{scriptsize}
Specializing to integer values for $x$, we get the identity matrix for $x=0$, and for $x=1$ we get the triangle of (unsigned) Stirling numbers  of the first kind \seqnum{A132393} which begins
$$\left(
\begin{array}{ccccccc}
 1 & 0 & 0 & 0 & 0 & 0 & 0 \\
 0 & 1 & 0 & 0 & 0 & 0 & 0 \\
 0 & 1 & 1 & 0 & 0 & 0 & 0 \\
 0 & 2 & 3 & 1 & 0 & 0 & 0 \\
 0 & 6 & 11 & 6 & 1 & 0 & 0 \\
 0 & 24 & 50 & 35 & 10 & 1 & 0 \\
 0 & 120 & 274 & 225 & 85 & 15 & 1 \\
\end{array}
\right).$$
For $x=2$ we get the triangle \seqnum{A079641} which begins
$$\left(
\begin{array}{ccccccc}
 1 & 0 & 0 & 0 & 0 & 0 & 0 \\
 0 & 1 & 0 & 0 & 0 & 0 & 0 \\
 0 & 2 & 1 & 0 & 0 & 0 & 0 \\
 0 & 6 & 6 & 1 & 0 & 0 & 0 \\
 0 & 26 & 36 & 12 & 1 & 0 & 0 \\
 0 & 150 & 250 & 120 & 20 & 1 & 0 \\
 0 & 1082 & 2040 & 1230 & 300 & 30 & 1 \\
\end{array}
\right).$$ This is the product of Riordan arrays
$$[1, e^x-1] \cdot \left[1, \ln\left(\frac{1}{1-x}\right)\right]$$
where the first matrix is the triangle of Stirling numbers of the second kind and the second matrix is the triangle of unsigned Stirling numbers of the first kind. This is \seqnum{A008277}. In the Del\'eham notation, this is
$$ [0, 2, 1, 4, 2, 6, 3, 8, 4, 10, 5, \ldots ]\, \Delta\, [1, 0, 1, 0, 1, 0, 1, 0, 1, 0, \ldots].$$

For $x=3$ we obtain the matrix that begins
$$\left(
\begin{array}{ccccccc}
 1 & 0 & 0 & 0 & 0 & 0 & 0 \\
 0 & 1 & 0 & 0 & 0 & 0 & 0 \\
 0 & 3 & 1 & 0 & 0 & 0 & 0 \\
 0 & 12 & 9 & 1 & 0 & 0 & 0 \\
 0 & 66 & 75 & 18 & 1 & 0 & 0 \\
 0 & 480 & 690 & 255 & 30 & 1 & 0 \\
 0 & 4368 & 7290 & 3555 & 645 & 45 & 1 \\
\end{array}
\right),$$ which is the product
$$\left[1, \frac{e^{2x}-1}{2}\right]\cdot  \left[1, \ln\left(\frac{1}{1-x}\right)\right].$$

Since $$A_n^{(k)}(x)=k^n F_n(x,1/k)$$ we arrive at the following formula for the $1/k$-Eulerian polynomials.
$$A_n^{(k)}(x)=\sum_{m=0}^n \sum_{j=0}^n S_{n,j}(x-1)^{n-j} |s_{j,m}|k^{n-m}.$$

\section{The shifted $1/k$-Eulerian polynomials}
We have seen that the shifted $1/k$-Eulerian polynomials $A_n^{(k)}(x)$, which begin
$$1, 1, kx + 1, k^2 x^2 + kx(k + 3) + 1,\ldots,$$ are the moments for a family of orthogonal polynomials. We now turn our attention to the once-shifted sequence $A_{n+1}^{(k)}(x)$, which begins
$$ 1, kx + 1, k^2 x^2 + kx(k + 3) + 1,\ldots.$$ Its generating function is
$$\frac{d}{dz} \left(\frac{1-x}{e^{kz(x-1)}-x}\right)^{1/k}=e^{kz(x-1)} \left(\frac{1-x}{e^{kz(x-1)}-x}\right)^{1/k+1}.$$
\begin{proposition} The once-shifted $1/k$-Eulerian polynomials $A_{n+1}^{(k)}(x)$ are the moments for the family of orthogonal polynomials whose coefficient array is given by the exponential Riordan array
$$\left[\frac{1}{(1+kxz)(1+kz)^{1/k}}, \frac{1}{k(1-x)} \ln\left(\frac{1+kz}{1+kxz}\right)\right].$$ The once-shifted $1/k$-Eulerian polynomials $A_{n+1}^{(k)}(x)$ form the initial column of the inverse matrix
$$\left[e^{kz(x-1)} \left(\frac{1-x}{e^{kz(x-1)}-x}\right)^{1/k+1}, \frac{e^{kz}-e^{kxz}}{k(e^{kxz}-x e^{kz})}\right].$$
\end{proposition}
\begin{proof}
We calculate $Z(z)$ and $A(z)$ for the Riordan array $\left[e^{kz(x-1)} \left(\frac{1-x}{e^{kz(x-1)}-x}\right)^{1/k+1}, \frac{e^{kz}-e^{kxz}}{k(e^{kxz}-x e^{kz})}\right]$.
This gives
$$Z(z)=kx(z(k+1)+1)+1, \quad\quad A(z)=(1+kz)(1+kxz).$$
Thus the production matrix is tri-diagonal, as required.
\end{proof}
\begin{corollary} The once-shifted $1/k$-Eulerian polynomials $A_{n+1}^{(k)}(x)$ have their ordinary generating function given by the continued fraction
$$\cfrac{1}{1-(kx+1)z-
\cfrac{kx(k+1)z^2}{1-(2kx+k+1)z-
\cfrac{2kx(2k+1)z}{1-(3kx+2k+1)z-
\cfrac{3kx(3k+1)z}{1-(4kx+3k+1)z-\cdots}}}}.$$
\end{corollary}

We now turn our attention to the sequence $F_n(x,y)$, which begins
$$1, y, y(x + y), y(x^2 + x(3y + 1) + y^2), y(x^3 + x^2(7y + 4) + x(6y^2 + 4y + 1) + y^3),\ldots.$$
In this case, we consider the sequence $F_{n+1}(x,y)/y$, which begins
$$1, (x + y), (x^2 + x(3y + 1) + y^2), (x^3 + x^2(7y + 4) + x(6y^2 + 4y + 1) + y^3),\ldots.$$
The generating function of this sequence is
$$\frac{1}{y}\frac{d}{dz} \left(\frac{1-x}{e^{z(x-1)}-x}\right)^y=e^{z(x+y)}\left(\frac{1-x}{e^{z(x-1)}-x}\right)^{y+1}.$$
We then have the following proposition.
\begin{proposition}
The polynomials $F_{n+1}(x,y)/y$ are the moments for the family of orthogonal polynomials whose coefficient array is the exponential Riordan array
$$\left[\frac{1}{(1+xz)(1+z)^y}, \frac{1}{1-x} \ln\left(\frac{1+z}{1+xz}\right)\right].$$
These polynomials appear as the initial column of the inverse array
$$\left[e^{z(x+y)}\left(\frac{1-x}{e^{z(x-1)}-x}\right)^{y+1}, \frac{e^{z}-e^{xz}}{e^{xz}-x e^{z}}\right].$$
\end{proposition}
\begin{proof} We calculate $Z(z)$ and $A(z)$ for the moment array
$\left[e^{z(x+y)}\left(\frac{1-x}{e^{z(x-1)}-x}\right)^{y+1}, \frac{e^{z}-e^{xz}}{e^{xz}-x e^{z}}\right]$.
We obtain
$$Z(z)=x(z(1+y)+1)+y,\quad\quad A(z)=(1+z)(1+xz).$$
We conclude that the production matrix of the moment matrix is tri-diagonal, as required.
\end{proof}
\begin{corollary} The once-shifted Savage-Viswanathan polynomials have their ordinary generating function given by the continued fraction
$$\cfrac{1}{1-(x+y)z-
\cfrac{x(y+1)z^2}{1-(2x+y+1)z-
\cfrac{2x(y+2)z^2}{1-(3x+y+2)z-
\cfrac{3x(y+3)z^2}{1-(4x+y+3)z-\cdots}}}}.$$
\end{corollary}

\section{Sample sequences}
The sequence \seqnum{A226513}, based on \cite{Ahlbach}, documents the array of sequences $F_n(2,y)$, $y \in \mathbb{N}$. These sequences count the number of barred preferential arrangements of $k$ things with $n$ bars ($k \ge 0$, $n \ge 0$). The sequence $F_n(2,1)$ is \seqnum{A000670}, the Fubini numbers (number of preferential arrangements of $n$ labeled elements). The sequence $F_n(2,2)$ is \seqnum{A005649}, the self-convolution of the preferential arrangements (this counts the number of compatible bi-partitional relations on a set of cardinality $n$). The sequence $F_n(2,3)$ is \seqnum{A226515}. The sequence $F_n(2,4)$ is \seqnum{A226738}.

The sequence $F_n(1,k)$ is given by $\frac{(n+k-1)!}{(k-1)!}$. These sequences are \seqnum{A000007}, \seqnum{A000142}, \seqnum{A000142}$(n+1)$, \seqnum{A001710}$(n+2)$, \seqnum{A001715} for $k=0,1,2,3,4$ respectively.

\bigskip
\hrule
\bigskip
\noindent 2010 {\it Mathematics Subject Classification}: Primary
11B83; Secondary 33C45, 42C-5, 15B36, 15B05, 11C20.
\noindent \emph{Keywords:} Eulerian polynomial, exponential Riordan array, orthogonal polynomial, moment sequence

\bigskip
\hrule
\bigskip
\noindent (Concerned with sequences
\seqnum{A000007},
\seqnum{A000142},
\seqnum{A000670},
\seqnum{A001710},
\seqnum{A001715},
\seqnum{A005649},
\seqnum{A007318},
\seqnum{A008277},
\seqnum{A079641},
\seqnum{A084938},
\seqnum{A132393},
\seqnum{A226513},
\seqnum{A226515} and
\seqnum{A226738}.)

\end{document}